\documentclass[12pt]{amsart}
\usepackage{epsfig,color}

\makeatother

\theoremstyle{definition}

\def\fnum{equation}
\newtheorem{Thm}[\fnum]{Theorem}
\newtheorem{Cor}[\fnum]{Corollary}

\newtheorem{Lem}[\fnum]{Lemma}
\newtheorem{Con}[\fnum]{Conjecture}

\numberwithin{equation}{section}

\newcommand{\Vol}{{\text{Vol}}}
\newcommand{\V}{{\text{V}}}

\newcommand{\Ric}{{\text{Ric}}}

\newcommand{\Tr}{{\text{Tr}}}

\newcommand{\Hess}{{\text {Hess}}}
\def\ZZ{{\bold Z}}
\def\RR{{\bold R}}

\def\SS{{\bold S}}

\newcommand{\e}{{\text {e}}}

\newcommand{\cL}{{\mathcal{L}}}
\newcommand{\cH}{{\mathcal{H}}}

\newcommand{\cN}{{\mathcal{N}}}

\newcommand{\cM}{{\mathcal{M}}}

\newcommand{\cP}{{\mathcal{P}}}

\newcommand{\eqr}[1]{(\ref{#1})}

 \newenvironment{dedication}
        {\vspace{6ex}\begin{quotation}\begin{center}\begin{em}}
        {\par\end{em}\end{center}\end{quotation}}

\def\bH{{\bold H}}

\title[Liouville properties]{Liouville properties}

\author[]{Tobias Holck Colding}%
\address{MIT, Dept. of Math.\\
77 Massachusetts Avenue, Cambridge, MA 02139-4307.}
\author[]{William P. Minicozzi II}%
 
\thanks{The  authors
were partially supported by NSF Grants DMS 1812142 and DMS 1707270.}

\email{colding@math.mit.edu  and minicozz@math.mit.edu}

\begin{document}

\maketitle

\begin{dedication}
To S.T. Yau on his seventieth birthday.
\end{dedication}

\begin{abstract}
The classical Liouville theorem states that a bounded  harmonic function on all of $\RR^n$ must be constant.  In the early 1970s, S.T. Yau vastly generalized this, showing that it holds
 for manifolds with nonnegative Ricci curvature.  Moreover, he conjectured a stronger Liouville property that has generated many significant developments.  We will first discuss this conjecture and some of the ideas 
 that went into its proof.
 
We will also discuss two recent  areas where this circle of ideas has played a major role.  One is Kleiner's new proof of Gromov's classification of groups of polynomial growth and the developments this generated.  Another is to understanding singularities of mean curvature flow in high codimension.   We will see that some of the ideas discussed in this survey naturally lead  to a new approach to studying and classifying singularities of mean curvature flow in higher codimension.  This is a subject that has been notoriously difficult and where much less is known than for hypersurfaces.  
\end{abstract}

\section{Introduction}

The classical Liouville theorem, named after Joseph Liouville  (1809 - 1882), states that a bounded (or even just positive) harmonic function on all of $\RR^n$ must be constant.   There is a very short proof of this for  bounded functions using the mean value property:  

\begin{quote}
Given two points, choose two balls with the given points as centers and of equal radius. If the radius is large enough, the two balls will coincide except for an arbitrarily small proportion of their volume. Since the function is bounded, the averages of it over the two balls are arbitrarily close, and so the function assumes the same value at any two points.
\end{quote}

The Liouville theorem has had a huge impact across many fields, such as complex analysis, partial differential equations, geometry, probability, discrete mathematics and   complex and algebraic geometry.  As well as many applied areas.  The impact of the Liouville theorem has been even larger as the starting point of many further developments.  

On manifolds with nonegative Ricci curvature, mean values inequalities hold,  but are no longer equalities,  and the above proof does not give a Liouville type property.  However, in the 1970s, S.T. Yau, \cite{Ya1}, showed that the Liouville theorem holds for such manifolds.  Later, in the mid 1970s, Yau together with S.Y. Cheng, \cite{CgYa}, showed a gradient estimate on these manifolds giving an effective version of the Liouville theorem; see also, Schoen, \cite{Sc}.  

The situation is very different for negatively curved manifolds such as hyperbolic space.  This is easiest seen in two dimensions where being harmonic is conformally invariant, so each harmonic function on the Euclidean disk is also harmonic in the hyperbolic metric. 
In particular,  each continuous function on the circle extends to a harmonic function on the disk and the space of bounded harmonic functions is infinite dimensional; cf. Anderson, \cite{A}, Sullivan, \cite{S}, and Anderson-Schoen, \cite{ASc}. 

In general, given a complete manifold $M$ and a  nonnegative constant $d$,  $\cH_d(M)$ is the linear space of harmonic functions of polynomial growth at most $d$:
\begin{quote}
That is, $u \in \cH_d (M)$ if $\Delta u = 0$ and for some $p \in M$ and a constant $C_u$ 
$$
	\sup_{B_R(p)} |u| \leq C_u\, (1 + R)^d {\text{ for all }} R \, .
$$
\end{quote}

In 1974, S.T. Yau conjectured the following stronger Liouville property: 
  
  \begin{Con}	\label{c:yau}
  If $M^n$ has $\Ric \geq 0$, then $\cH_d (M)$ is finite dimensional for each $d$.
  \end{Con}
  
 This conjecture generated many significant developments and was discussed by many authors.  See, for instance:
   page $117$ in \cite{Ya3}, problem $48$ in \cite{Ya4}, Conjecture $2.5$  in \cite{Sc}, \cite{Ka1}, \cite{Ka2}, \cite{Kz}, \cite{DF}, Conjecture $1$ in \cite{Li1}, and problem (1) in \cite{LiTa1}, amongst others.
  
  \vskip4mm
  The conjecture  was settled in \cite{CM4}:
  
  \begin{Thm}  \label{t:cm14}
\cite{CM4} Conjecture \ref{c:yau} holds.  
\end{Thm}

In fact, \cite{CM4}  proved finite dimensionality under much weaker assumptions of:
\begin{enumerate}
\item  A volume doubling bound.
\item A scale-invariant Poincar\'e inequality.  
\end{enumerate}
Both (1) and (2) hold for $\Ric \geq 0$ by the Bishop-Gromov volume comparison and \cite{B}.  However, these properties
  do not require much regularity of the space and are quite flexible.  In particular, they make sense for more general metric-measure spaces and are preserved by bi-Lipschitz changes of the metric.  Moreover, the properties (1) and (2) make sense also for discrete spaces, vastly extending the theory and methods out of the continuous world.  This extension opens up   applications to geometric group theory and discrete mathematics; some of which we will touch upon later.

\section{Harmonic polynomials on Euclidean space}

There are two simple ways to understand $\cH_d (\RR^n)$.  The first, which is very special to Euclidean space, uses that the Laplacian commutes with partial derivatives on $\RR^n$.
 The key   is then the
gradient estimate{\footnote{The constant $\sqrt{2n+16}$ is not sharp.}}:
\begin{quote}
\item If $\Delta u = 0$ on $B_{2R} \subset \RR^n$, then 
$$
	\sup_{B_R} |\nabla u| \leq \frac{\sqrt{2n+16}}{R} \, \sup_{B_{2R}} |u| \, .
$$
\end{quote}
Thus, 
  if $|u| \leq C$ on all of $\RR^n$, then $\sup_{B_R} |\nabla u| \leq \frac{C\, \sqrt{2n+16}}{R}$ for all $R$.  Letting $R \to \infty$ gives that $u$ is constant.
 Since $\partial_{x_i} \Delta u = \Delta (\partial_{x_i} u)$ on $\RR^n$, the gradient estimate implies that if $u \in \cH_d (\RR^n)$, then $\frac{\partial u}{\partial x_i} \in \cH_{d-1} (\RR^n)$.  Applying this $d$ times gives that the $d$-th order partials are constant and, thus, $u$ is a polynomial of degree $d$.
 It follows that $\cH_d (\RR^n)$ 
 is the space of harmonic polynomials of degree at most $d$ and, thus, 
has dimension   of the order $d^{n-1}$.  
 
There is another way to think of this that gives a more general perspective.  Namely, 
in polar
coordinates $(\rho , \theta) \in \RR^+ \times \SS^{n-1}$,   the Laplacian is
\begin{equation}
\Delta_{\RR^n}     = \rho^{-2} \Delta_{\theta}
        + (n-1)\,\rho^{-1} \frac{\partial}{\partial \rho}
                + \frac{\partial^2}{\partial \rho^2}  \, .
\end{equation}
In particular, the restriction of a homogeneous harmonic
polynomial of degree $d$ to $\SS^{n-1}$ gives an eigenfunction
with eigenvalue $d^2 + (n-2)d$.  

  A similar ``cone construction'' holds more generally; cf. \cite{CM2}.
 Given a  manifold $N^{n-1}$,  the {\it cone} over $N$ is
 the manifold $C(N) = N \times [0,\infty)$ with the metric
\begin{equation}
    ds_{C(N)}^2 = dr^2 + r^2 \, ds^2_N \, .
\end{equation}
  The Laplacians of $N$ and $C(N)$
are related by 
\begin{equation}                \label{e:sepv}
        \Delta_{C(N)} u   = r^{-2} \Delta_N
                u
        + (n-1)\,r^{-1} \frac{\partial}{\partial r} u
                + \frac{\partial^2}{\partial r^2} u \, .
\end{equation}
Using \eqr{e:sepv}, we can now reinterpret the spaces $\cH_d
(C(N))$:

\begin{Lem}   \label{l:equiv}
  If $\Delta_N g = - \lambda \, g$   on $N^{n-1}$, then
  $r^d \, g \in \cH_d (C(N))$ where
  \begin{equation}
    d^2 +       (n-2) d = \lambda \, .
\end{equation}
\end{Lem}

As a consequence of  Lemma \ref{l:equiv},  the spectral properties
of $N$ are equivalent to properties of harmonic functions of
polynomial growth on   $C(N)$.  In this way, the dimension bounds on $\cH_d$, for $d$ large, are related to spectral asymptotics on the cross-section which are given by Weyl's asymptotic formula.  This point of view was a focus point of \cite{CM5}.

      \section{Laplacian on a manifold}
 
 On a Riemannian manifold $M$ with metric $\langle \cdot , \cdot \rangle$ and Levi-Civita connection $\nabla$, the gradient of a function $f$ is defined by
 \begin{equation}
 	V(f) = \langle \nabla f , V \rangle {\text{ for all vectors fields }} V \, .
\end{equation}
The  Laplacian of $f$ is the trace of the Hessian.  That is, if $e_i$ is an orthonormal frame for $M$, then
 \begin{equation}
 	\Delta f  = \Tr \, \Hess_f =  \sum_i \Hess_f (e_i , e_i) = \sum_i \langle \nabla_{e_i} V , e_i \rangle \, .
\end{equation}

  For harmonic functions,  we get the following reverse Poincar\'e inequality (also sometimes called the Caccioppoli inequality):
  
 \begin{Lem}[Reverse Poincar\'e]	\label{l:rp}
 If $\Delta u =0$ on $B_{2R} \subset M$, then
 \begin{equation}
 	\int_{B_R} |\nabla u|^2 \leq  \frac{4}{R^2} \, \int_{B_{2R}} u^2 \, .
 \end{equation}
\end{Lem}

Incidentally, Yau used this in \cite{Ya2} to show:

  \begin{Thm}[Yau]
If $M$ is open, $u$ is harmonic, and $\int u^2 < \infty$, then $u$ must be constant.  If $M$ has infinite volume, then $u \equiv 0$.
\end{Thm}

\subsection{The Bochner formula}

On Euclidean space $\RR^n$, partial derivatives commute and 
\begin{equation}	\label{e:tribo}
	\frac{1}{2} \, \Delta_{\RR^n} |\nabla u|^2 = \left|\Hess_u \right|^2 + \langle \nabla u , \nabla \Delta_{\RR^n} u \rangle \, .
\end{equation}
On a Riemannian manifold $M$,  derivatives do not commute and \eqr{e:tribo} does not hold in general.  However, the failure of derivatives to commute is measured by the Riemann curvature tensor.    Using this, 
Bochner proved the following extremely useful formula:
\begin{equation}
\frac{1}{2}\Delta |\nabla u|^2 =|\Hess_u|^2+ \langle \nabla u , \nabla \Delta_{M} u \rangle  +\Ric_M(\nabla
u,\nabla u)\, .
\end{equation}
Thus, if $\Ric_M \geq 0$,  then the energy density of a harmonic function is subharmonic.

\section{Gradient estimate for harmonic functions}

Gradient estimates  have played a key role in geometry and PDE 
since at least the early work of Bernstein in the 1910s.  These are probably
the most fundamental a priori estimates for elliptic and parabolic
equations, leading to Harnack inequalities, Liouville theorems,
and compactness theorems for both linear and nonlinear PDE.

 A typical example for
linear equations is the well--known, and highly influential, gradient estimate of S.Y.
Cheng and S.T. Yau for harmonic functions:

\begin{Thm}  \cite{CgYa}   \label{t:cy}
If $\Delta u = 0 $   on  $B_R(0)$ with nonnegative Ricci
curvature, then
\begin{equation}    \label{e:cy}
    |\nabla u|(0) \leq \frac{\sqrt{2\,n+16}}{ R} \, \, 
   \sup_{B_R} |u|  \, .
\end{equation}
 \end{Thm}

\begin{proof}  
We begin by introducing a cutoff function $\eta$ that vanishes on $\partial B_R$ and has bounds on its gradient and Laplacian.
Define the cutoff function $\eta (x) = R^2 - r^2$, where $r$ is the distance function to the center $0$ of the ball.  Observe that
\begin{align}
	\nabla \eta &= |2\,r \, \nabla r| \leq 2\, R \, ,  \\
	|\nabla \eta^2|  &= 2\, \eta \, |\nabla \eta| \leq 4 \, R \, \eta \, , \\
	\Delta \eta &= - \Delta r^2 \geq - 2\,n \, ,  \\
	\Delta \eta^2 &= 2\, |\nabla \eta |^2 + 2\, \eta \, \Delta \eta \geq - 4\,n \eta \geq - 4\,n \, R^2 \, , 
\end{align}
where the third line used the Laplacian comparison theorem (which applies since $\Ric \geq 0$).

 Using the product rule the Laplacian, the Bochner formula and the above formulas for $\eta$, 
 we compute that
\begin{align}    \label{e:cy1}
    \Delta (\eta^2 \, |\nabla u|^2)  &= 
    \eta^2 \, \Delta |\nabla u|^2 + 2\, \langle \nabla \eta^2 , \nabla |\nabla u|^2 \rangle + |\nabla u|^2 \, \Delta \eta^2 \notag
    \\
    &\geq    2 \, \eta^2 \, |\Hess_u
    |^2 - 16 \, R \, 
    \, \eta \, |\nabla u| \, |\Hess_u|
    -4\,n\, R^2 \, |\nabla u|^2   \notag \\ &\geq - (4\,n+32) \, R^2 \, |\nabla u|^2 \, ,
\end{align}
where the last inequality used the absorbing inequality $16\,a\,b \leq
2\,a^2 + 32\, b^2$ with $a= \eta \, |\Hess_u|$ and $b= R \, |\nabla u|$.
On the other hand, we have 
$
	\Delta u^2 = 2 \, |\nabla u|^2 \, , 
$
so the function 
\begin{equation}
w = (2\,n+16) \, R^2 \, u^2 +
\eta^2 \, |\nabla u|^2
\end{equation}
 is subharmonic on $B_R(0)$ (i.e., $\Delta
w \geq 0$).  By the maximum principle, the maximum of $w$ occurs
on the boundary so that
\begin{equation}    \label{e:cy2}
     R^4 \, |\nabla u|^2 (0) \leq w(0) \leq \max_{\partial B_R } w =
     (2\,n+16) \, R^2 \, \max_{\partial B_R} u^2  \, .
\end{equation}
\end{proof}

In fact, Cheng and Yau proved a stronger estimate: 

\begin{Thm}[Gradient estimate; \cite{CgYa}] \label{t:cy2}
If $\Delta u = 0$ and 
$u$ is  positive on $B_{R}(0)$ with $\Ric \geq 0$, then
\begin{equation}     \label{e:pge}
   \frac{|\nabla u|}{u} \, (0)=  |\nabla \log u| (0) \leq \frac{4\,n}{R}  \, .
\end{equation}
\end{Thm}

 An important
consequence of \eqr{e:pge} is the Harnack inequality for positive harmonic
functions.  Accordingly, estimates of the form \eqr{e:pge}, which give a bound for the derivative of the logarithm of a positive function, are often referred to as differential Harnack estimates.
In 1986, Li and Yau proved a sharp gradient, or differential Harnack, estimate for the heat equation on manifolds with $\Ric \geq 0$,  \cite{LiY}.    The paper \cite{C2} gives a sharp elliptic gradient estimate on such manifolds.  Finally, note that the Harnack inequality holds more generally for manifolds with a volume doubling and Poincar\'e inequality by \cite{Gr}, \cite{SC}.

\section{Harmonic functions with polynomial growth on general spaces}

The Cheng-Yau gradient estimate  implies the global Liouville theorem of
Yau, \cite{Ya3}, by taking $R \to \infty$ in \eqr{e:pge}.
In fact,  it implies the stronger result that any harmonic function with sublinear growth
must be constant:

\begin{Cor}
\cite{CgYa} If $M$ is complete with $\Ric_M \geq 0$,
then $\cH_d (M) = \{ {\text{Constant functions}} \}$ for  $d<1$.
\end{Cor}

Since $\RR^n$ has nonnegative Ricci curvature and the coordinate
functions are harmonic, this is obviously sharp. Therefore, when
$d\geq 1$ a different approach is needed.  Instead of
showing a Liouville theorem, the point is to control the size of
the space of solutions.  Over the years, there were many
interesting partial results (including   when $M$ is a
surface, \cite{LiTa2} and
\cite{DF}).  In \cite{LiTa1}, P. Li and L.F. Tam
obtained the borderline case $d=1$, showing that
\begin{equation}    \label{e:linearg}
    \dim (\cH_1 (M)) \leq n+1 \, ,
\end{equation}
 for an $n$-dimensional manifold with $\Ric_M \geq 0$. When $M=\RR^n$ the space $\cH_1 (\RR^n)$ has dimension  $n+1$ and is spanned by the $n$ coordinate functions plus
the constant functions.  The corresponding rigidity theorem was
proven in \cite{ChCM} (see \cite{Li2} for the special case where
$M$ is K\"ahler):

\begin{Thm}     \label{t:chcm}
\cite{ChCM}
 If $M$ is complete with $\Ric_M \geq 0$, then every
tangent cone at infinity $M_{\infty}$ splits isometrically as
\begin{equation}
    M_{\infty} = N \times \RR^{ \dim (\cH_1 (M)) - 1 } \, .
\end{equation}
  Hence, if $\dim (\cH_1 (M)) = n+1$, then \cite{C1} implies
that $M = \RR^n$.
\end{Thm}

Finally, in \cite{CM4},  Yau's conjecture from 1974 was settled.   Namely, \cite{CM4} showed that the spaces of polynomial growth harmonic
functions are finite dimensional; see Theorem \ref{t:cm14}.
 The proof  
consists of  two independent steps (the first does not use
harmonicity):
\begin{itemize} \item
Given a $2k$-dimensional subspace $H \subset \cH_d (M)$ and $h \in
(0,1]$, there exists a $k$-dimensional subspace $K \subset H$ and
$R> 0$ so that
\begin{equation}        \label{e:step1}
    \sup_{  v \in K \setminus \{ 0\} } \, \, \frac{\int_{ B_{(1+h)^2 R} }
    v^2}{\int_{ B_{R} }
    v^2} \leq C_1 \, (1+h)^{8d} \, .
\end{equation}
  \item
The dimension of a space $K$ of harmonic functions
satisfying \eqr{e:step1} is bounded in terms of $h$ and $d$.
\end{itemize}

To give some feel for the argument, we will sketch a proof of the
second step.

\begin{proof}(Sketch of second step)
 For simplicity, suppose that $R=1$ and $h=1$.  Fix a scale $r \in (0,1)$
 to be chosen small.
  We will use two properties of manifolds with $\Ric_M \geq 0$:

   First,
  we can find $N \leq
 C_n
 \, r^{-n}$  balls $B_r(x_i)$ with
 \begin{equation}       \label{e:vd}
    \chi_{B_1} \leq \sum_i \chi_{ B_{r} (x_i) } \leq C_n \,
    \chi_{B_2} \, ,
 \end{equation}
where $\chi_E$ is the characteristic function of a set $E$. (To do
this,  choose a maximal disjoint collection of balls of radius
$r/2$ and then use the volume comparison to get the second
inequality in \eqr{e:vd} and bound   $N$.)

Second, there is a uniform Poincar\'e inequality: If
$\int_{B_s(x)} f = 0$, then
\begin{equation}    \label{e:np}
    \int_{B_s(x)} f^2 \leq C_N \, s^2  \int_{B_s(x)} |\nabla f|^2
    \, .
\end{equation}

To bound the dimension of $K$, we will  construct a linear map
$\cM : K \to \RR^N$  and show that $\cM$ is injective for $r>0$
sufficiently small.  We define $\cM$ by
\begin{equation}
    \cM (v) = \left( \int_{B_{r}(x_1)} v \, ,  \cdots , \int_{B_{r}(x_N)}
    v  \right) \, .
\end{equation}
We will deduce a contradiction if $v \in K \setminus \{ 0 \}$ is
in the kernel of $\cM$. In particular, \eqr{e:np} gives that for
each $i$
\begin{equation}    \label{e:np2}
    \int_{B_r(x_i)} v^2 \leq C_N \, r^2  \int_{B_r(x_i)} |\nabla v|^2
    \, .
\end{equation}
Combining this with \eqr{e:vd} gives
\begin{equation}    \label{e:np3}
    \int_{B_1} v^2 \leq \sum_{i=1}^N \int_{B_r(x_i)} v^2 \leq C_N \, r^2 \sum_{i=1}^N \int_{B_r(x_i)} |\nabla v|^2
    \leq C_n \, C_N \, r^2 \,  \int_{B_2} |\nabla v|^2
    \, .
\end{equation}
We now (for the only time) use that $v$ is harmonic.  Namely, the
Caccioppoli inequality (or reverse Poincar\'e inequality) for
harmonic functions gives
\begin{equation}    \label{e:rp}
    \int_{B_2} |\nabla v|^2 \leq \int_{B_4} v^2 \, .
\end{equation}
Combining \eqr{e:np3} and \eqr{e:rp}, we get
\begin{equation}    \label{e:np4}
    \int_{B_1} v^2 \leq  C_n \, C_N \, r^2 \,  \int_{B_4} v^2
    \, .
\end{equation}
This contradicts \eqr{e:step1} if $r$ is sufficiently small,
completing the proof.
\end{proof}

On Euclidean space $\RR^n$, the spaces $\cH_d$ are given by
harmonic polynomials of degree at most $d$.  In particular, it is
not hard to see that
\begin{equation}    \label{e:ruppern}
    \dim (\cH_d (\RR^n)) \approx
    C \, d^{n-1} \, .
\end{equation}
 Using the correspondence between
harmonic polynomials and eigenfunctions on $\SS^{n-1}$ (see Lemma
\ref{l:equiv}), this
 is   closely related to Weyl's asymptotic formula on $\SS^{n-1}$.
 In \cite{CM5}, the authors proved a
similar sharp polynomial bound for manifolds with non--negative
Ricci curvature:

\begin{Thm}     \label{t:cm15}
\cite{CM5}  If $M^n$ is complete with $\Ric_M \geq 0$ and $d\geq
1$, then
\begin{equation}    \label{e:cm15}
    \dim ( \cH_d (M)) \leq C \, d^{n-1} \, .
\end{equation}
\end{Thm}

Taking $M = \RR^n$,   \eqr{e:ruppern}  illustrates that the
exponent $n-1$ is sharp in \eqr{e:cm15}.  However, as in Weyl's
asymptotic formula, the constant in front of $d^{n-1}$ can be
related to the volume. Namely,   we actually showed the stronger
statement
\begin{equation}    \label{e:cm15a}
    \dim ( \cH_d (M)) \leq C_n \, \V_M \, d^{n-1} + o (d^{n-1}) \,
    ,
\end{equation}
where
\begin{itemize} \item $C_n$ depends only on the
dimension $n$. \item
  $\V_M$ is the ``asymptotic volume ratio'' $\lim_{r\to \infty} \,
  \Vol (B_r)/ r^n$. \item $o(d^{n-1})$ is a function of $d$ with
  $\lim_{d\to \infty} \, o(d^{n-1})/d^{n-1} = 0$.
  \end{itemize}
As noted above, Theorem \ref{t:cm15} also gives lower bounds for
eigenvalues on a manifold $N^{n-1}$ with $\Ric_N \geq (n-2) =
\Ric_{\SS^{n-1}}$.  Using the sharper estimate \eqr{e:cm15a}
introduces the volume of $N$ into these eigenvalue estimates (as
predicted by Weyl's asymptotic formula).

An interesting  feature of these dimension estimates is that they
follow from ``rough'' properties of $M$ and are therefore
surprisingly stable under perturbation. For instance, in
\cite{CM4}, we proved Theorem \ref{t:cm14} for manifolds
with a volume doubling and a Poincar\'e inequality; unlike
a Ricci curvature bound, these properties are stable under
bi--Lipschitz transformations.

This finite dimensionality was not previously known even for
manifolds bi--Lipschitz to $\RR^n$ (except under additional
hypotheses, cf.  Avellenada--Lin, \cite{AvLn}, and
Moser--Struwe, \cite{MrSt}).

\vskip2mm
The volume doubling and  Poincar\'e inequality together imply 
  a meanvalue inequality.  Using the meanvalue inequality and the doubling, we prove  finite dimensionality for harmonic sections of certain bundles in \cite{CM6} (see also \cite{CM3}).

\vskip2mm
This is just a very brief overview (omitting many interesting
results), but we hope that it gives something of the flavor of the
subject; see
 \cite{CM3} and references
therein for more.

\section{The heat equation}

A function $u$ satisfies the heat equation if $u_t = \Delta u$.  In particular, harmonic functions are static (time-independent) solutions of the heat equation.

\subsection{Ancient solutions of the heat equation}

The natural parabolic generalization of a polynomial growth harmonic function is a polynomial growth ancient solution of the heat equation.  On $\RR^n$, it is classical that these are just polynomials in $x$ and $t$ and, thus, finite dimensional.  In view of \cite{CM4}, \cite{CM5} and \cite{CM6} for harmonic functions, it is  natural to seek dimension bounds for these spaces on manifolds.  This was initiated by Calle in  2006 in her thesis, \cite{Ca1}, \cite{Ca2}; 
cf. \cite{LZ} for some recent results, including a parabolic generalization of \cite{CM4}.

Given $d>0$, let $\cP_d (\RR^n)$ be the space of ancient solutions $u(x,t)$ of the heat equation $\partial_t u = \Delta\, u$ so that there exists $C_u$ with
\begin{align}
	\sup_{ B_R \times [-R^2,0]}\, |u| \leq C_u \, \left( 1 + R \right)^d \, .
\end{align}

\begin{Thm}	\label{t:CMPd1}
\cite{CM11}
If   $\Vol (B_R(p)) \leq C \, (1+R)^{d_V}$ for some $p \in M$, then for $1\leq k \in \ZZ$
\begin{align}
	\dim \cP_{2k}(M) \leq (k+1) \, \dim \cH_{2k}(M) \, .
\end{align}
\end{Thm}

Combining this
with the bound $   \dim \cH_d (M) \leq C \, d^{n-1}$ when  $\Ric_{M^n} \geq 0$  from \cite{CM5} gives:
 
\begin{Cor}	\label{t:CMPd}
\cite{CM11}
There exists $C=C(n)$ so that if   $\Ric_{M^n} \geq 0$, then for $d\geq 1$
\begin{align}	\label{e:cmpd}
	\dim \, \cP_d (M) \leq C \, d^n \, .
\end{align}
\end{Cor}

 This result is optimal in the following sense:  There is a constant $c=c_n$  so that for $d\geq 1$
\begin{align}	\label{e:cpdrn}
	c^{-1} \, d^n \leq \dim \cP_d (\RR^n) \leq c \, d^n \, .
\end{align}
Thus, 
 the exponent $n$ in \eqr{e:cmpd}   is  sharp; see Lin and Zhang, 
\cite{LZ}, for a recent related result that adapts the methods of \cite{CM4}--\cite{CM6} to get the weaker bound $d^{n+1}$.

 Theorem \ref{t:CMPd1}  gives finite dimensionality of $\cP_d (M)$ for any $M$ where $\cH_d(M)$ is finite dimensional.  Thus, the earlier results of \cite{CM4}--\cite{CM6} give finite dimensionality of $\cP_d (M)$ 
 when $M$ has a volume doubling and either a Poincar\'e or meanvalue inequality.

\subsection{Parabolic gradient estimates}

The next lemma gives a simple interior gradient estimate that parallels the gradient estimate of Theorem \ref{t:cy} for harmonic functions.  One can also get gradient estimates in time, though the scaling factor is different.

\begin{Lem}	\label{l:hge}
If $M$ is complete with $\Ric_M \geq 0$, the there exists $C$ depending on $n$ so that if $(\partial_t -\Delta)\, u=0$ on $B_R \times [-R^2,0]$, then
\begin{align}
	\sup_{B_{R/2} \times [-R^2/4,0]} \, |\nabla u|^2 \leq \frac{C}{R^2} \, \sup_{ B_R \times [-R^2,0]} u^2 \, .
\end{align}
\end{Lem}

\begin{proof}
By scaling, it suffices to prove the estimate when $R=1$.  Let $\psi$ be a cut-off function that is one on $B_{1/2} \times [-1/4,0]$ and zero on the parabolic  boundary of 
$B_1 \times [-1,0]$.  Note that
$(\partial_t - \Delta) u^2  = - 2 \, |\nabla u|^2$ and, by the Bochner formula since $\Ric \geq 0$, 
	$(\partial_t - \Delta) |\nabla u|^2  \leq - 2 \, |\nabla^2 u|^2$.
Therefore, the Kato inequality $|\nabla |\nabla u|| \leq |\nabla^2 u|$ and the absorbing inequality give  
\begin{align}
	(\partial_t - \Delta)\, \left[ \psi^2  |\nabla u|^2 \right] &\leq - 2 \, |\nabla^2 u|^2 + |\nabla u|^2 \, (\partial_t - \Delta) \psi^2 + 8 |u| \, |\nabla u| |\psi| |\nabla \psi| \notag \\
	&\leq |\nabla u|^2 \,\left\{  (\partial_t - \Delta) \psi^2 +  8\,  |\nabla \psi |^2 \right\} \, .
\end{align}
Using the Laplacian comparison theorem, we can construct $\psi$ so that
 $  \left|  (\partial_t - \Delta) \psi^2 +  8\,  |\nabla \psi |^2 \right| \leq C$ for a constant $C$ depending just on $n$.  It follows that $C \, u^2 + |\nabla u|^2 \, \psi^2$ is a subsolution of the heat equation and the parabolic maximum principle gives
\begin{align}
	\sup_{B_{1/2} \times [-1/4,0]} \, |\nabla u|^2 \leq C \, \sup_{ B_1 \times [-1,0]} u^2 \, .
\end{align}
After rescaling to radius $R$, this gives the lemma.
\end{proof}

In \cite{LiY}, Li and Yau proved a gradient estimate for positive solutions of the heat equation:

\begin{Thm}[Differential Harnack inequality; \cite{LiY}] \label{t:LY}
If $\partial_t u = \Delta u = 0$ and 
$u$ is  positive on $M \times [0,\infty )$ with $\Ric \geq 0$, then
\begin{equation}     
  \frac{|\nabla u|^2}{u^2} - \frac{u_t}{u}  \leq \frac{n}{2t}  \, .
\end{equation}
\end{Thm}

There is also a local version of Theorem \ref{t:LY} in \cite{LiY} when $u$ is positive on $B_R \times [-R^2,0]$ with $\Ric \geq 0$.  Namely, there exists $C_n$ depending on $n$ so that
\begin{equation}     \label{e:localLY}
  \sup_{B_{\frac{R}{2}} \times [-\frac{R^2}{4} , 0]} \, \, \left( \frac{|\nabla u|^2}{u^2} - \frac{u_t}{u} \right) \leq \frac{C_n}{ R^2}  \, .
\end{equation}

An immediate corollary of \eqr{e:localLY} is a  generalization of Lemma \ref{l:hge}:

\begin{Cor}	\label{c:LY}
There exists $C$ depending on $n$ so that 
if $\partial_t u = \Delta u = 0$ and 
  $\Ric \geq 0$, then
\begin{equation}     \label{e:localLY2}
  \sup_{B_{\frac{R}{2}} \times [-\frac{R^2}{4} , 0]} \, \,  |u_t|   \leq \frac{C}{ R^2} \, \sup_{ B_R \times [-R^2,0] } \, \, |u|  \, .
\end{equation}
Similarly, we have $ \sup_{B_{\frac{R}{2}} \times [-\frac{R^2}{4} , 0]} \, \,  |\nabla u|^2   \leq \frac{C}{ R^2} \, \sup_{ B_R \times [-R^2,0] } \, \, u^2$.
\end{Cor}

\begin{proof}
Let $m$ be the supremum of $|u|$ on $ B_R \times [-R^2,0] $.  Then $v=u+m$ and $w= m -u$ satisfy the heat equation and are positive with
\begin{align}
	0 \leq v,w \leq 2m \, .
\end{align}
Let $\Omega = B_{\frac{R}{2}} \times [-\frac{R^2}{4} , 0]$.
Applying \eqr{e:localLY} to $v$ gives on $\Omega$ that
\begin{equation}     
	 \frac{|\nabla u|^2}{v^2} - \frac{u_t}{v}  =   \frac{|\nabla v|^2}{v^2} - \frac{v_t}{v}  \leq \frac{C_n}{ R^2}  \, .
\end{equation}
Thus, we get on $\Omega$ that $u_t \geq -\frac{C_n\, v}{R^2} \geq - \frac{2C_n \, m}{R^2}$.  Applying \eqr{e:localLY} to $w$ gives on $\Omega$ that
\begin{equation}     \label{e:putbackinhere}
	 \frac{|\nabla u|^2}{w^2} + \frac{u_t}{w}  =   \frac{|\nabla w|^2}{w^2} - \frac{w_t}{w}  \leq \frac{C_n}{ R^2}  \, ,
\end{equation}
which gives that $u_t \leq \frac{C_n \, w}{R^2} \leq \frac{2C_n \, m}{R^2}$.    Combining the upper and lower bounds on $u_t$ gives \eqr{e:localLY2}.  Finally, 
using \eqr{e:localLY2}
in \eqr{e:putbackinhere} gives the spatial gradient estimate.
\end{proof}

Using the parabolic gradient estimate of Li-Yau in place of the Cheng-Yau gradient estimate, we will get generalizations of the harmonic rigidity theorems when the degree of growth is low.

\begin{Cor}	\label{c:pdlow}
Suppose that $\Ric_{M^n} \geq 0$.  If $d<1$, 
then $\cP_d (M) = \{ {\text{constant functions}} \}$.  If $d<2$, then $\cP_d (M) = \cH_d (M)$.  Finally, 
$\dim \cP_1 (M) \leq n+1$ with equality  only on $\RR^n$
\end{Cor}

 \begin{proof}
Suppose that $u \in \cP_d(M)$ with $d<1$.
By taking $R \to \infty$ in Lemma \ref{l:hge}, we see that $|\nabla u| \equiv 0$ and, thus, $u$ is constant in space.  The equation $u_t = \Delta u$ then implies that $u$ is also constant in time.

 By Corollary \ref{c:LY}, we see that if $d < 2$, then $\cP_d(M) = \cH_d(M)$.   In particular, 
$\dim \cP_1 (M) \leq n+1$, by \cite{LiTa1}, and equality holds if and only if $M = \RR^n$ by \cite{ChCM}.
\end{proof}

\section{Recent results}

In the last two sections we will discuss two recent applications of the methods discussed here in two very different directions.  The first is a new proof, by Kleiner, of Gromov's theorem about groups of polynomial growth whereas the second, which is discussed in the next section, is to blow-ups (ancient solutions) of curvature flow.

\subsection{Connections with geometric group theory}

Recently Kleiner, \cite{K}, (see also Shalom-Tao, \cite{ST}) used, in part, the circle of ideas discussed here in his new proof of an important and foundational result in geometric group theory, originally due to Gromov, \cite{G}.   Gromov's theorem asserts that any finitely generated group of polynomial growth has a finite index nilpotent subgroup. 

Given an infinite group generated by a finite symmetric set, a function on the group is said to be harmonic if it obeys the mean value equality.   Here the mean is taken over adjacent elements.  Kleiner's proof has roughly four steps, cf. \cite{T1}, \cite{T2}.     The first is to construct plenty of polynomial growth harmonic functions on any group with polynomial growth.     The second step uses that the space of polynomial growth functions $\cH_d$ on the group is finite dimensional for each $d$.  The third step shows that any finitely generated group of polynomial growth that sits inside a compact Lie subgroup of the general linear group is virtually abelian.  Finally, the  fourth step uses an induction argument to reduce the general question to the third step.  Steps one and two together give  that step three applies.    To get the key finite dimensionality of the second step, Kleiner shows a Poincar\'e inequality and observes that the group satisfies a type of doubling condition.  

\section{A new approach to MCF in higher codimension}

We will see that ideas discussed in this survey naturally lead to a new approach to studying and classifying singularities of mean curvature flow (MCF) in higher codimension.  This is a subject that has been notoriously difficult and where much less is known than for hypersurfaces.   
The idea of \cite{CM10} is to use ideas described in the earlier sections to show that blowups of higher codimension MCF have codimension that typically is  much smaller than in the original flow.   In many important instances we can show that blowups are evolving hypersurfaces in an Euclidean subspace even when the original flow is very far from being hypersurfaces.     

A one-parameter family of $n$-dimensional submanifolds $M^n_{t}\subset \RR^N$ evolves by {\emph{mean curvature flow}} if each point $x(t)$ evolves by 
\begin{align}
\partial_{t} x =   -  \bH\, ,
\end{align}
where $\bH = -\Tr \, A$ is the mean curvature vector and $A$ is the second fundamental form.  It  is said to be {\emph{ancient}} if it exists for all negative times.    
The restrictions of the coordinate functions on $\RR^N$  to the evolving submanifolds satisfy the heat equation.  This  connects the study of MCF with the study of spaces of polynomial growth functions.  
Indeed one way of thinking about MCF is that the position vector $x\in M^n_t\subset \RR^N$ satisfies the nonlinear heat equation
\begin{align}
(\partial_t-\Delta_{M_t})\,x=0\, .
\end{align}
This equation is nonlinear since the Laplacian depends on the evolving submanifold $M_t$. 

There is a Lyapunov function for the flow that is particularly useful.  To define it recall that the Gaussian surface area $F$ of an $n$-dimensional submanifold
 $\Sigma^n \subset \RR^N$ is
\begin{align}
	F(\Sigma) = \left( 4\,\pi \right)^{ - \frac{n}{2}} \,  \int_{\Sigma} \e^{ - \frac{|x|^2}{4} } \, .
\end{align}
The factor $ \left( 4\, \pi \right)^{ - \frac{n}{2} }$ is chosen to make the Gaussian area one for  an $n$-plane through the origin.
Following \cite{CM9}, the entropy $\lambda$ is the supremum of $F$ over all translations and dilations
\begin{align}
	\lambda (\Sigma) = \sup_{c,x_0} \, F (c\,\Sigma + x_0) \, .
\end{align}
By Huisken's monotonicity, \cite{H}, it follows that $\lambda$ is monotone nonincreasing under the flow.  From this, and lower semi continuity of $\lambda$, 
we have that all blowups have entropy bounded by that of the initial submanifold in a MCF.   

 MCF  in higher codimension is a   complicated nonlinear parabolic system and much less is known than for hypersurfaces.
 The singularities are modeled by shrinkers $\Sigma$ that evolve by scaling.  Shrinkers
get more  complicated as the codimension increases.   
    The most fundamental shrinkers are cylinders $\SS^k_{\sqrt{2k}} \times \RR^{n-k}$,
but there are many others including all $n$-dimensional minimal submanifolds of the sphere $\partial B_{\sqrt{2n}} \subset \RR^N$.  
The entropy of round spheres is monotone decreasing in the dimension, \cite{St}, with
\begin{align}
\sqrt{2}\leq \lambda ({\SS}^n)<\lambda (\SS^{n-1})<\cdots<\lambda (\SS^1)=\sqrt{\frac{2\,\pi}{\e}}\approx 1.52\, ,
\end{align}
and $\lambda (\Sigma\times \RR)=\lambda (\Sigma)$.   

Any blowup of a MCF leads to an ancient flow.  A particularly important way of blowing up is around a fixed point in space-time.  This kind of ancient flow is a shrinker.  A submanifold $\Sigma$ is a shrinker if it satisfies the equation
\begin{align}
	\bH = \frac{x^{\perp}}{2} \, ,
\end{align}
where $x^{\perp}$ is the perpendicular part of the position vector field.  This is equivalent to saying that the one parameter family $\sqrt{-t}\,\Sigma$ flows by the MCF.  

\subsection{Bounding codimension}       Let $M_t^n \subset \RR^N$ be an ancient MCF with entropies 
$\lambda (M_t)\leq \lambda_0<\infty$.  The  space  $\cP_d$ consists of polynomial growth caloric functions $u(x,t)$ on $\cup_t M_t\times\{t\}$ so that  $(\partial_t - \Delta_{M_t})\, u=0$ and there exists $C$ depending on $u$ with
  \begin{align}
  	|u(x,t)| \leq C\, (1 + |x|^d+|t|^{\frac{d}{2}}) {\text{ for all }} (x,t) \text{ with }x\in M_t \, .
  \end{align}
  
  In \cite{CM10}, for each $d$ we bound the dimension of $\cP_d(M_t)$ for an ancient MCF $M_t\subset \RR^N$.  
  The bound is in terms of the dimension of $M_t$, the entropy, and $d$.
The next result is a special case of this for $d=1$ that shows that the codimension of the smallest Euclidean space that the flow sits inside is bounded in terms of the entropy.

\begin{Thm}	\label{t:counting2}
(Bounding codimension by entropy for ancient flows, \cite{CM10}).  If $M_t^n\subset \RR^N$ is an ancient MCF, then $M_t\subset$ an Euclidean subspace of dimension $\leq C_n\, \sup_t \lambda (M_t)$.
\end{Thm}

There is a parallel of this result that can been seen as a generalization of a well-known result of Cheng-Li-Yau.  To explain this let $\Sigma^n \subset \RR^N$ be a shrinker with finite entropy $\lambda (\Sigma)$. 
    We will use $\| u \|_{L^2}$ to denote the Gaussian $L^2$ norm.
As in \cite{CM6}, the drift Laplacian (Ornstein-Uhlenbeck operator) $\cL = \Delta - \frac{1}{2} \nabla_{x^T}$ is self-adjoint with respect to the Gaussian inner product
$
	\int_{\Sigma} u\, v \, \e^{ -  \frac{|x|^2}{4} }  \, .
$
We will say that $u$ is a $\mu$-eigenfunction   if $\cL \, u = - \mu \, u$   and $0 < \| u \|_{L^2} < \infty$. 
    The {\it spectral counting function} $\cN (\mu)$ is the number of eigenvalues $\mu_i\leq \mu$ counted with multiplicity.   In \cite{CM10} we bound the spectral 
    counting function for any shrinker in terms of $n$, $\lambda(\Sigma)$, and $\mu$.  
        As a special case we get:
    	
  \begin{Thm}	\label{c:counting}
  (Bounding codimension by entropy for shrinkers, \cite{CM10}).  
  If $\Sigma^n \subset \RR^N$ is a shrinker, then $\Sigma\subset$ an Euclidean subspace of dimension $\leq C_n \, \lambda(\Sigma)$.
  \end{Thm}

Our estimates in Theorems \ref{t:counting2} and \ref{c:counting} are linear in the entropy, which is known to be sharp.  The corresponding linear estimate   for algebraic varieties in
 complex projective space follows from B\'ezout's theorem, $18.3$ in \cite{Ha}.    
 When $\Sigma \subset \partial B_{\sqrt{2\,n}}\subset \RR^N$ is a closed $n$-dimensional minimal submanifold of the sphere and the entropy reduces to the volume,  then this estimate follows from theorem $6$ in \cite{CLY}.

  In theorem $1.5$ in \cite{dCW}, do Carmo and Wallach construct families of minimal submanifolds of the sphere, each  isometric to the same round sphere, 
 generalizing earlier results of Calabi  \cite{Ca}.  The boundary immersions of the families in \cite{dCW} lie in a lower-dimensional affine  space.   Obviously, they have the same
  volume and, since they are contained in spheres, also the same entropy.  Thus,  the number of linearly independent coordinate functions can vary along a family.

   \subsection{Sharp bound for codimension}
Suppose that $M_t^n\subset \RR^N$ is an ancient MCF with $\sup_t\lambda (M_t)<\infty$.  For each constant $c>0$ define the flow $M_{c,t}$ by
\begin{align}
M_{c,t}=\frac{1}{c}\,M_{c^{2}\,t}\, .
\end{align}
It follows that $M_{c,t}$ is an ancient MCF as well.  Since  $\sup_t\lambda(M_t)<\infty$, Huisken's monotonicity, \cite{Hu}, and work of Ilmanen, \cite{I}, White, \cite{Wh1}, gives that every sequence $c_i\to \infty$ has a subsequence (also denoted by $c_i$) so that $M_{c_i,t}$ converges to a shrinker $M_{\infty,t}$ (so $M_{\infty,t}=\sqrt{-t}\,M_{\infty,-1}$) with $\sup_t\lambda (M_{\infty,t})\leq \sup_t\lambda (M_t)$.  We will say that such a $M_{\infty,t}$ is a tangent flow at $-\infty$ of the original flow.

The next result gives a sharp bound for the codimension:

\begin{Thm} \label{t:ancientcylinder}
(Sharp codimension bound, \cite{CM10}).  
  If $M^n_t\subset\RR^N$ is an ancient MCF and one tangent flow at $-\infty$ is a cylinder, then  $M_t$ is a flow of hypersurfaces in a Euclidean subspace.  
  \end{Thm}

We believe that this theorem will have wide ranging consequences for MCF in higher codimension.   To briefly explain some of these we make the following conjecture:  
 
  \begin{Con}  \label{c:bigC}
  For $n\leq 4$ and any codimension, round generalized cylinders, $\SS^{k}_{\sqrt{2\,k}}\times \RR^{n-k}$, are the shrinkers with the lowest entropy. 
  \end{Con}

We also conjecture that for any $n$ the round $\SS^n$ has the least entropy of any closed shrinker $\Sigma^n \subset \RR^N$.  
 The corresponding result for hypersurfaces was proven in \cite{CIMW}; see also \cite{HW}.     It was also noted in \cite{CIMW} that the ``Simons cone'' over $\SS^2\times \SS^2$ has entropy strictly less than that of $\SS^1\times \RR^4$.  In other words, 
already for $n=5$, the round generalized cylinders  is not a complete list of the lowest entropy shrinkers.
Conjecture \ref{c:bigC} is known for $n=1$ since shrinking curves  are  contained in  affine two-planes and have entropy at least that of round circles.  It is also known to hold for $n=2$ and $N=3$ by work of Bernstein-Wang, \cite{BW}.  

If Conjecture \ref{c:bigC} holds, then combined with Theorem \ref{t:ancientcylinder} it would follow that any ancient flow $M_t^n \subset \RR^N$ with entropy at most $\lambda (\SS^1)$ plus some  small $\epsilon > 0$ would be a hypersurface in some Euclidean subspace of dimension $n+1$ provided $n\leq 4$.  This would give that all blowups near any cylindrical singularity for $n\leq 4$ are ancient flows of hypersurfaces.    Thus, reducing the system to a single differential equation.

\end{document}